\documentclass[a4paper,12pt]{article}

\usepackage{amsmath, amssymb}
\usepackage{amsthm}
\usepackage{graphicx,color}
\usepackage{hyperref}
\usepackage{rawfonts}

\input prepictex
\input pictex
\input postpictex

\newcommand{\carka}{,\penalty0\relax}

\newcommand{\zlom}{\penalty0\relax}

\newtheorem{theorem}{Theorem}

\newtheorem{lemma}[theorem]{Lemma}
\newtheorem{proposition}[theorem]{Proposition}
\newtheorem{corollary}[theorem]{Corollary}

\def\email#1{\href{mailto:#1}{\texttt{#1}}}
\def\NP{\ensuremath{\mathcal{NP}}}
\def\Z{\ensuremath{\mathbb{Z}}}
\def\N{\ensuremath{\mathbb{N}}}

\title{The Packing Coloring of Distance Graphs $D(k,t)$}
\author{Jan Ekstein$^{*}$ ~~~ P\v{r}emysl Holub\thanks{University of West Bohemia, Pilsen, Czech Republic,\newline e-mail:          
        \email{$\lbrace$ekstein, holubpre$\rbrace$@kma.zcu.cz}.}\and
        Olivier Togni\thanks{University of Bourgogne, Dijon, France,\newline e-mail: \email{Olivier.Togni@u-bourgogne.fr}.}}
\date{\today}

\begin{document}

 \maketitle

 \bigskip

 \begin{abstract}
  The packing chromatic number $\chi_{\rho}(G)$ of a graph $G$ is the smallest integer $p$ such that vertices of $G$ can be 
  partitioned into disjoint classes $X_{1}, ..., X_{p}$ where vertices in $X_{i}$ have pairwise distance greater than $i$. For 
  $k < t$ we study the packing chromatic number of infinite distance graphs $D(k, t)$, i.e. graphs with the set $\Z$ of integers as 
  vertex set and in which two distinct vertices $i, j \in \Z$ are adjacent if and only if $|i - j| \in \{k, t\}$.
  
  We generalize results by Ekstein et al. for graphs $D (1, t)$. For sufficiently large $t$ we prove that $\chi_{\rho}(D(k, t)) \leq 30$
  for both $k$, $t$ odd, and that $\chi_{\rho}(D(k, t)) \leq 56$ for exactly one of $k$, $t$ odd. We also give some upper and lower 
  bounds for $\chi_{\rho}(D(k, t))$ with small $k$ and $t$.
 \end{abstract}

\section{Introduction}
 The concept of a packing coloring was introduced by Goddard et al. \cite{God} under the name broadcast coloring where an application to 
 frequency assignments was indicated. In a given network the signals of two stations that are using the same broadcast frequency will 
 interfere unless they are located sufficiently far apart. The distance, in which the signals will propagate, is directly related to the
 power of those signals. Bre\v sar et al. in \cite{Bresar} mentioned that this concept could have several additional applications, as, 
 for instance, in resource placements and biological diversity (different species in a certain area require different amounts of
 territory). Moreover, the concept is both a packing and a coloring (i.e., a partitioning) concept. Therefore Bre\v sar et al. in 
 \cite{Bresar} proposed a notion {\em packing coloring} which we follow in this paper.      
 
 In this paper we consider simple undirected graphs only. For terminology and notations not defined here we refer to
 \cite{Bon}. Let $G$ be a connected graph and let $\mbox{dist}_{G}(u, v)$ denote the distance between vertices $u$ and
 $v$ in $G$. We ask for a partition of the vertex set of $G$ into disjoint classes $X_{1}, ..., X_{p}$ according to the
 following constraints. Each color class $X_{i}$ should be an \emph{$i$-packing}, a set of vertices with property that
 any distinct pair $u, v \in X_{i}$ satisfies $\mbox{dist}_{G}(u, v) > i$. Such a partition is called a
 \emph{packing $p$-coloring}, even though it is allowed that some sets $X_{i}$ may be empty. The smallest integer $p$
 for which there exists a packing $p$-coloring of $G$ is called the \emph{packing chromatic number} of $G$ and it is
 denoted $\chi_{\rho}(G)$. The determination of the packing chromatic number is computationally difficult. It was shown to 
 be \NP-complete for general graphs in~\cite{God}. Fiala and Golovach \cite{FiaGol} showed that the problem remains \NP-complete 
 even for trees.

 Let $D = \{d_{1}, d_{2}, ..., d_{k}\}$, where $d_{i}$ ($i = 1, 2, ..., k$) are positive integers such that
 $d_{1} < d_{2} < ... < d_{k}$. The (infinite) \emph{distance graph} $D(d_{1}, d_{2}, ..., d_{k})$ has the set $\mathbb{Z}$ of integers
 as a vertex set and in which two distinct vertices $i, j \in \mathbb{Z}$ are adjacent if and only if $|i - j| \in D$. The study of 
 a coloring of distance graphs was initiated by Eggleton et al. \cite{Eggl} and a lot of papers concerning this topic have been published
 (see \cite{ChaHuZhu}, \cite{KemKo}, \cite{LiLaSo}, \cite{RuzTuVoi}, \cite{Zhu} for a sample of results). 
 
 The study of a packing coloring of distance graphs was initiated by Togni. In \cite{Togni} Togni showed that 
 $\chi_{\rho}(D(1, t)) \leq 40$ for odd $t \geq 447$ and that $\chi_{\rho}(D(1, t)) \leq 81$ for even $t \geq 448$. Ekstein et al. 
 in \cite{Ekstein} improved these upper bounds and proved that $\chi_{\rho}(D(1, t)) \leq 35$ for odd $t \geq 575$ and that 
 $\chi_{\rho}(D(1, t)) \leq 56$ for even $t \geq 648$. A lower bound 12 for the packing chromatic number of $D(1, t)$, for $t \geq 9$,
 is also given in \cite{Ekstein}.
  
 In this paper we generalize mentioned results as follows. 
 
 \bigskip
 
 \begin{theorem}
  \label{k, t odd}
   Let $k, t$ be odd positive integers such that $t \geq 825$ and $k, t$ coprime (i. e. $D(k, t)$ is a~connected distance graph). 
   Then $$\chi_{\rho}(D(k, t)) \leq 30.$$
 \end{theorem}
  
 Note that, for $k = 1$, Theorem \ref{k, t odd} also improves the upper bound given in~\cite{Ekstein}.
 
 \begin{corollary}
  For any odd positive integer $t\geq 825$, $\chi_{\rho}(D(1,t))\leq 30$.
 \end{corollary}
  
 \begin{theorem}  
  \label{k odd, t even} 
   Let $k, t$ be positive integers such that $k$ is odd, $t \geq 898$ is even and $k, t$ coprime
   (i.e. $D(k, t)$ is a~connected distance graph). Then $$\chi_{\rho}(D(k, t)) \leq 56.$$  
 \end{theorem}

 \begin{theorem}
  \label{k even, t odd} 
   Let $k, t$ be positive integers such that $k$ is even, $t \geq 923$ is odd and $k, t$ coprime
   (i. e. $D(k, t)$ is a~connected distance graph). Then $$\chi_{\rho}(D(k, t)) \leq 56.$$  
 \end{theorem}
 
 \begin{theorem}
  \label{general lower bound}
  Let $D(k, t)$ be a connected distance graph, $t \geq 9$. Then $$\chi_{\rho}(D(k, t)) \geq 12.$$
 \end{theorem}
 
 For $k$, $t$ both even and also for $k, t$ commensurable, the distance graph $D(k, t)$ is disconnected and contains copies of
 a distance graph $D(k', t')$ as its components with $k' < k$, $t' < t$, at least one of $k', t'$ odd and $k', t'$ coprime 
 (as will be shown in a proof of Lemma \ref{coprime}). In the view of this fact, we can color each copy of $D(k',t')$ in the same way, thus we obtain the following statement.
 
 \begin{proposition}
  Let $k,t$ be positive integers and $g$ their greatest common divisor. 
  Then $\chi_{\rho}(D(k,t))=\chi_{\rho}\left(D(\frac kg, \frac tg)\right)$.
 \end{proposition}
 
 For small values of $k, t$ we give lower and upper bounds as it is shown in the following Table~\ref{tb1}. Note that, for $k = 1$, 
 an analogous table was published in \cite{Ekstein}.
 
 \bigskip
 
 \begin{table}[ht]
  \centering
  \footnotesize
  \begin{tabular}{|c||c|c|c|c|c|c|c|c|}
  \hline
  k $\diagdown$ t  & 3 & 4 & 5 & 6 & 7 & 8 & 9 & 10 \\
  \hline
  \hline
  2 & {\bf 13} & $D(1,2)$ & $14- 22 $ & $D(1,3)$ & $15- 27 $ & $D(1,4)$ & $12- 31$ & $D(1,5)$ \\
  \hline
  3  & --- & $14- 19 $ & {\bf 13} & $D(1,2)$ & $13- 17$ & $14- 28 $ & $D(1,3)$& $13 - 29 $ \\
  \hline
  %\hline
  4  & --- & --- & $13- 22 $ & $D(2,3)$ & $16 - 32 $ & $D(1,2)$ & $15- 32$ & $D(2,5)$ \\
  \hline
  5 & --- & --- & --- & $15 - 29 $ & $13 - 20$ & $14- 32 $ & $13- 23$ & $D(1,2)$ \\
  \hline
  %\hline
  6 & --- & --- & --- & --- & $15- 29 $& $D(3,4)$ & $D(2,3)$ & $D(3,5)$ \\
  \hline
  7 & --- & --- & --- & --- & --- & $14- 34 $& $12- 23 $& $12- 40$\\
  \hline
  %\hline
  8 & --- & --- & --- & --- & --- & --- & $12- 37 $ & $D(4,5)$ \\
  \hline
  9 & --- & --- & --- & --- & --- & --- & --- & $12- 42$\\
  \hline
  \end{tabular}
  \caption{\label{tb1} Values and bounds of $\chi_{\rho}(D(k,t))$ for $2 \le k < t \le 10$. The emphasized numbers are exact values and 
  all pairs of values are lower and upper bounds.}
 \end{table}

 Throughout the rest of the paper by a coloring we mean a packing \zlom coloring.

%%%%%%%%%%%%%%%%%%%%%%%%%%%%%%%%%%%%%%%%%%% Small t
 \section{Bounds for $\chi_{\rho}(D(k, t))$ for small $k$, $t$}
 \label{Small t}
 In this section we determine new lower and upper bounds for the packing chromatic number of $D(k, t)$  which are mentioned in   
 Table~\ref{tb1}.

 For the upper bounds, we found and verified (with a help of a computer) patterns, which can be periodically repeated for a whole 
 distance graph $D(k, t)$. This means that we color vertices $1, \dots, l$ of $D(k,t)$ using a pattern of length $l$ and copy this 
 pattern on vertices $1+pl, \dots, l+pl$, $p \in \mathbb{Z}$. As most of these patterns are of big lengths, they do not appear in this 
 paper, but can be viewed at the web page {\tt http://le2i.cnrs.fr/o.togni/packdist/} and tested using the java applet provider.
 
 \begin{table}[ht]
  \centering
  $$ \begin{array}{|c|c|c|c|c|}
  \hline
  k,t & c & p & \text{ Configurations } & \text{ Time }\\ \hline
  2,3  & 12 & 213 &  1.1\cdot10^{12} & 46 \text{ hours} \\ \hline
  2,5  & 13 & 45 & 5.9\cdot10^{12} &  327 \text{ hours} \\ \hline
  3,4  & 13 &  43 & 5.6\cdot10^{12} & 297 \text{ hours}\\ \hline
  3,5  & 12 &  106 & 5.7\cdot10^{11} & 35 \text{ hours}\\ \hline
  3,7  & 12 &  54 & 2\cdot10^{12} & 179 \text{ hours}\\ \hline
  4,5  & 12 &  37 & 4\cdot10^{11} & 23 \text{ hours}\\ \hline
  \end{array}$$
  \caption{\label{tb2} Computations for finding lower bounds of $\chi_{\rho}(D(k,t))$. Time of the computation is measured on a one-core 
                       workstation from year 2012.}
 \end{table}

 \bigskip
 
 \begin{table}[ht]
  \centering
  $$ \begin{array}{|c|c|c|c|c|}
  \hline
  k,t & q & b & \text{ Configurations } & \text{ Time }\\ \hline
  2,7  & 5 &  37/45  & 32.7\cdot10^9 &  14 \text{ min}\\ \hline
  3,8  & 7 &  38/42  &  1.5\cdot10^{11} & 1 \text{ hours}\\ \hline
  3,10 & 8 &  47/50  &  10^{13} & 125\text{ hours}\\ \hline
  4,7  & 7 &  44/50  &  7.6\cdot10^{12} & 58\text{ hours}\\ \hline
  4,9  & 7 &  47/52  &  1.9\cdot10^{13} & 145\text{ hours}\\ \hline
  5,6  & 6 &  43/50  &  1.1\cdot10^{12} & 7\text{ hours}\\ \hline
  5,7  & 8 &  47/50  &  2.7\cdot10^{12} & 19\text{ hours}\\ \hline
  5,8  & 8 &  42/45  &  7.8\cdot10^{12} & 58\text{ hours}\\ \hline
  5,9  & 7 &  48/52  &  1.2\cdot10^{13} & 107\text{ hours}\\ \hline
  6,7  & 6 &  44/50  &  3.5\cdot10^{12} & 27\text{ hours}\\ \hline
  7,8  & 6 &  50/55  &  1.4\cdot10^{13} & 120\text{ hours}\\ \hline
  \end{array}$$
  \caption{\label{tb3} Computations for finding lower bounds of $\chi_{\rho}(D(k,t))$. Time of the computation is measured on a one-core 
                       workstation from year 2012.}
 \end{table}

 For the lower bounds, we followed methods used in proofs of Lemmas 6 and 8 in \cite{Ekstein}. Some of the lower bounds were obtained 
 using brute force search programs (one in Pascal and one in C++). We showed that a subgraph $D_p(k,t)$ of $D(k,t)$ induced by vertices 
 $\{1,2,\ldots,p\}$ cannot be colored using colors from 1 to $c$, for some $p$ and $c$ (the results of computations can be seen in 
 Table \ref{tb2}), which implies that $\chi_{\rho}(D(k,t))\geq c+1$. For a shortening of a computation time the programs precolored 
 vertex $1$ with color $c$ and tried to extend the coloring for whole $D_p(k,t)$.
 
 For the remaining lower bounds we used a density method. A density of a color class $X_i$ in a packing coloring of $G$ can be defined 
 as a fraction of all vertices of $X_i$ and all vertices of $G$. For the exact definition of densities $d(X_i)$ and 
 $d(X_1 \cup \dots \cup X_i)$ we refer to \cite{Ekstein}.
 
 The density method is based on the following proposition.
 
 \begin{lemma} \emph{\textbf{\cite{bib-fiala09+}}}
  \label{densities}
 For every finite packing coloring with $k$ classes $X_1,X_2, \dots ,X_k$ of a graph $G$ and any positive integer $l$ satisfying 
 $1\leq l\leq k$, it holds that

  $$\sum\limits_{i=1}^{k} d(X_i) \geq d(X_1 \cup \ldots \cup X_l) + \sum\limits_{i=l+1}^{k} d(X_i) \geq d(X_1 \cup \ldots \cup X_k) =1.$$

 \end{lemma}
 
 Again using brute force search programs (one in Pascal and one in C++) we showed that the density 
 $d(X_1 \cup X_2 \cup \dots \cup X_q)$ is bounded by $b$, for some $q$ and $b$ (the results are summarized in Table \ref{tb3}).
 For instance, for $k,t=2,7$ we found that $d(1,2,3,4)\le 32/41$. Since there is no pair of vertices in $D_{7i-13}(2,7)$ with distance 
 greater than $i$ ($i\geq 5$), $d(i)=1/(7i-13)$ for $i\geq 5$. Then $d(1,2,3,4) + d(5) + \cdots + d(13) \leq 0.9915326 < 1$ and hence,  
 from Lemma \ref{densities}, $\chi_{\rho}(D(2,7))\ge 14$. The bounds for the other values of $k,t$  in Table~\ref{tb3} are proved 
 similarly, observing that for $D(k,t)$, $d(i)=1/(ti -\alpha)$, with $\alpha = 13$ for $k,t=2,7$ and $i\ge 5$; $\alpha = 17$ for 
 $k,t=3,8$ and $i\ge 8$; $\alpha = 22$ for $k,t=3,10$ and $i\ge 9$; $\alpha = 8$ for $k,t=4,7$ and $i\ge 8$; $\alpha = 21$ for $k,t=4,9$ 
 and $i\ge 8$; $\alpha = 5$ for $k,t=5,7$ and $i\ge 9$; $\alpha = 10$ for $k,t=5,8$ and $i\ge 9$; $\alpha = 19$ for $k,t=5,9$ and 
 $i\ge 8$; and $\alpha = -1$ for $k=t-1$ and $i\ge t-1$.
 
 The lower bounds for $D(2,9)$, $D(7,9)$, $D(7, 10)$, $D(8, 9)$ and $D(9, 10)$ are obtained from Theorem \ref{general lower bound}.
 
%%%%%%%%%%%%%%%%%%%%%%%%%%%%%%%%%%%%%%%%%%% Large t
 \section{Proofs}
 \label{Large t}
  
  First of all we prove for which $k$, $t$ a distance graph $D(k, t)$ is connected.
 
 \begin{lemma}
  \label{coprime}
  A distance graph $D(k, t)$ is connected if and only if the greatest common divisor of $k$, $t$ is 1.
 \end{lemma} 
  
 \begin{proof}
  If the greatest common divisor of $k$, $t$ is 1, then from linear algebra $1 = mk + nt$, where $m, n \in \Z$. For every vertex 
  corresponding to a number $p \in \Z$ it holds that $p = pmk + pnt$ and therefore there exists a path between this vertex and vertex     
  corresponding to 0. Hence it follows that there exists a path between any pair of vertices and that $D(k, t)$ is connected.
  
  If the greatest common divisor of $k$, $t$ is $p > 1$, than we divide a set of vertices of $D(k, t)$ into subsets in terms of 
  equivalence classes modulo $p$ and clearly there is no edge between any pair of vertices from distinct subsets of vertices of
  $D(k, t)$. Hence $D(k, t)$ is not connected which completes the proof. Moreover, vertices corresponding to the equivalence  classes 
  modulo $p$ induce $p$ isomorphic copies of a graph $D(\frac{k}{p}, \frac{t}{p})$ where $\frac{k}{p}, \frac{t}{p}$ are coprime
  positive integers.
 \end{proof}  
 
 \begin{figure}[ht]
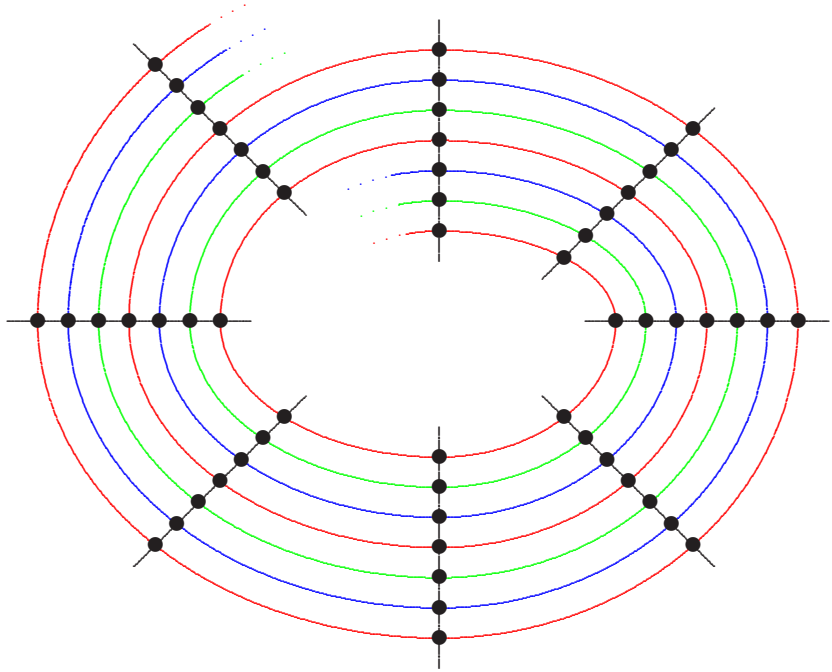

  $$\beginpicture
  \setcoordinatesystem units <.8mm,.8mm>
  \setplotarea x from 20 to 100, y from 10 to 30
  \plot 75 32.5 75 -7.5 /
  \plot 75 60 75 100 /
  \plot 99 50 139 50 /
  \plot 44 50 4 50 /
  \plot 91.97 57.07 120.25 85.36 /
  \plot 91.97 37.63 120.25 9.34 /
  \plot 24.8 9.34  53.08 37.63 /
  \plot 53.08 67.68  24.8 95.96 /
  {\color[rgb]{1,0,0}
  \ellipticalarc axes ratio 29:15 90 degrees from 104 50 center at 75 50
  \setdots \ellipticalarc axes ratio 21:15 35 degrees from 75 65 center at 75 50 \setsolid
  \ellipticalarc axes ratio 21:15 15 degrees from 75 65 center at 75 50
  \ellipticalarc axes ratio 29:22.5 90 degrees from 75 27.5 center at 75 50
  \ellipticalarc axes ratio 36:22.5 90 degrees from 39 50 center at 75 50
  \ellipticalarc axes ratio 36:30 90 degrees from 75 80 center at 75 50
  \ellipticalarc axes ratio 44:30 90 degrees from 119 50 center at 75 50
  \ellipticalarc axes ratio 44:37.5 90 degrees from 75 12.5 center at 75 50
  \ellipticalarc axes ratio 51:37.5 90 degrees from 24 50 center at 75 50
  \ellipticalarc axes ratio 51:45 90 degrees from 75 95 center at 75 50
  \ellipticalarc axes ratio 59:45 90 degrees from 134 50 center at 75 50
  \ellipticalarc axes ratio 59:52.5 90 degrees from 75 -2.5 center at 75 50 
  \setdots
  \ellipticalarc axes ratio 66:60 -62 degrees from 9 50 center at 75 50 
  \setsolid
  \ellipticalarc axes ratio 66:60 -55 degrees from 9 50 center at 75 50 
  \ellipticalarc axes ratio 66:52.5 90 degrees from 9 50 center at 75 50 
  }{\color[rgb]{0,1,0}
  \ellipticalarc axes ratio 34:20 90 degrees from 109 50 center at 75 50
  \setdots \ellipticalarc axes ratio 26:20 32 degrees from 75 70 center at 75 50 \setsolid
  \ellipticalarc axes ratio 26:20 15 degrees from 75 70 center at 75 50
  \ellipticalarc axes ratio 34:27.5 90 degrees from 75 22.5 center at 75 50
  \ellipticalarc axes ratio 41:27.5 90 degrees from 34 50 center at 75 50
  \ellipticalarc axes ratio 41:35 90 degrees from 75 85 center at 75 50
  \ellipticalarc axes ratio 49:35 90 degrees from 124 50 center at 75 50
  \ellipticalarc axes ratio 49:42.5 90 degrees from 75 7.5 center at 75 50
  \ellipticalarc axes ratio 56:42.5 90 degrees from 19 50 center at 75 50
  \ellipticalarc axes ratio 56:50 -55 degrees from 19 50 center at 75 50
  \setdots \ellipticalarc axes ratio 56:50 -63 degrees from 19 50 center at 75 50 \setsolid
  }{\color[rgb]{0,0,1}
  \ellipticalarc axes ratio 39:25 90 degrees from 114 50 center at 75 50
  \setdots \ellipticalarc axes ratio 31:25 29 degrees from 75 75 center at 75 50 \setsolid
  \ellipticalarc axes ratio 31:25 15 degrees from 75 75 center at 75 50
  \ellipticalarc axes ratio 39:32.5 90 degrees from 75 17.5 center at 75 50
  \ellipticalarc axes ratio 46:32.5 90 degrees from 29 50 center at 75 50
  \ellipticalarc axes ratio 46:40 90 degrees from 75 90 center at 75 50
  \ellipticalarc axes ratio 54:40 90 degrees from 129 50 center at 75 50
  \ellipticalarc axes ratio 54:47.5 90 degrees from 75 2.5 center at 75 50
  \ellipticalarc axes ratio 61:47.5 90 degrees from 14 50 center at 75 50
  \ellipticalarc axes ratio 61:55 -55 degrees from 14 50 center at 75 50
  \setdots \ellipticalarc axes ratio 61:55 -64 degrees from 14 50 center at 75 50 \setsolid
  }\setsolid

  \large
  \put{$\bullet$} at 75 65
  \put{$\bullet$} at 75 70
  \put{$\bullet$} at 75 75
  \put{$\bullet$} at 75 80
  \put{$\bullet$} at 75 85
  \put{$\bullet$} at 75 90
  \put{$\bullet$} at 75 95
  \put{$\bullet$} at 104 50
  \put{$\bullet$} at 109 50
  \put{$\bullet$} at 114 50
  \put{$\bullet$} at 119 50
  \put{$\bullet$} at 124 50
  \put{$\bullet$} at 129 50
  \put{$\bullet$} at 134 50
  \put{$\bullet$} at 75 27.5
  \put{$\bullet$} at 75 22.5
  \put{$\bullet$} at 75 17.5
  \put{$\bullet$} at 75 12.5
  \put{$\bullet$} at 75 7.5
  \put{$\bullet$} at 75 2.5
  \put{$\bullet$} at 75 -2.5
  \put{$\bullet$} at 39 50
  \put{$\bullet$} at 34 50
  \put{$\bullet$} at 29 50
  \put{$\bullet$} at 24 50
  \put{$\bullet$} at 19 50
  \put{$\bullet$} at 14 50
  \put{$\bullet$} at 9 50
  \put{$\bullet$} at 95.5 60.5
  \put{$\bullet$} at 95.5 34.09
  \put{$\bullet$} at 49.54 34.09
  \put{$\bullet$} at 49.54 71.21
  \put{$\bullet$} at 106.11 71.21
  \put{$\bullet$} at 106.11 23.48
  \put{$\bullet$} at 38.94 23.48
  \put{$\bullet$} at 38.94 81.82
  \put{$\bullet$} at 116.72 81.82
  \put{$\bullet$} at 116.72 12.88
  \put{$\bullet$} at 28.33 12.88
  \put{$\bullet$} at 99.04 64.14
  \put{$\bullet$} at 99.04 30.55
  \put{$\bullet$} at 46 30.55
  \put{$\bullet$} at 46 74.75
  \put{$\bullet$} at 109.65 74.75
  \put{$\bullet$} at 109.65 19.95
  \put{$\bullet$} at 35.4 19.95
  \put{$\bullet$} at 102.58 67.68
  \put{$\bullet$} at 102.58 27.02
  \put{$\bullet$} at 42.47 27.02
  \put{$\bullet$} at 42.47 78.28
  \put{$\bullet$} at 113.18 78.28
  \put{$\bullet$} at 113.18 16.41
  \put{$\bullet$} at 31.87 16.41
  \put{$\bullet$} at 28.33 92.42
  \put{$\bullet$} at 31.87 88.89
  \put{$\bullet$} at 35.34 85.35
  \endpicture$$
   \caption{Distance graph D(3, 8).}
   \label{D(3,8)}
 \end{figure}
 
 A key observation of this section is that a connected distance graph $D(k, t)$ can be drawn as $k$ vertex disjoint infinite  
 spirals with $t$ lines orthogonal to the spirals (e.g. $D(3, 8)$ on Fig. \ref{D(3,8)}).

 For $i \in \{0, 1, ..., t - 1\}$, the \emph{$i$-band} in a connected distance graph $D(k, t)$, denoted by $B_{i}$, is an infinite
 path in $D(k, t)$ on the vertices $V(B_{i}) = \{ik + jt, j \in \mathbb{Z}\}$. For $i \in \{0, 1, ..., t - 24\}$, the 
 \emph{$i$-strip} in a connected distance graph $D(k, t)$, $t > 24$, denoted by $S_{i}$, is a subgraph of $D(k, t)$ induced by the union
 of vertices of $B_{i}, B_{i + 1}, ..., B_{i + 23}$. 
 
 For a connected graph $D(k, t)$ we use a notation $D(k, t) = S_{0}B_{24}S_{25}B_{49}\ldots$ to express that we view $D(k, t)$ as 
 a union of strips $S_{0},S_{25},\ldots$ and bands $B_{24},B_{49},\ldots$ (including edges between strips and bands).
 
  \begin{figure}[ht]
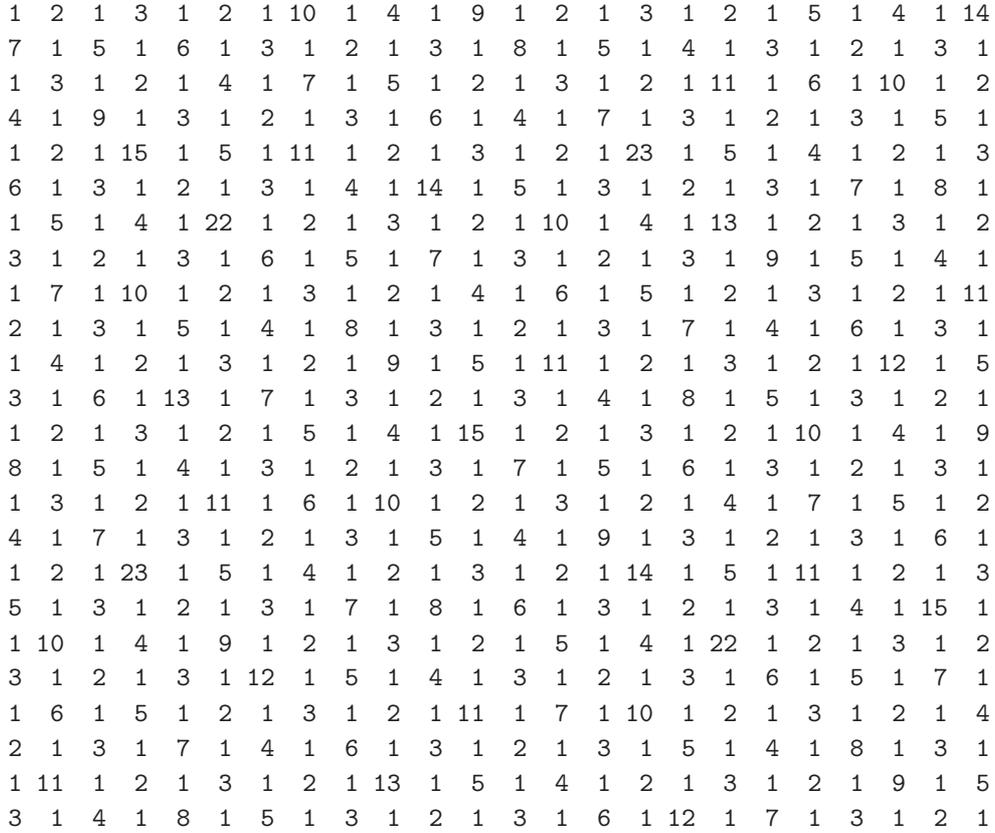

 \centering
 \footnotesize
 \begin{verbatim}
  1  2  1  3  1  2  1 10  1  4  1  9  1  2  1  3  1  2  1  5  1  4  1 14
  7  1  5  1  6  1  3  1  2  1  3  1  8  1  5  1  4  1  3  1  2  1  3  1
  1  3  1  2  1  4  1  7  1  5  1  2  1  3  1  2  1 11  1  6  1 10  1  2
  4  1  9  1  3  1  2  1  3  1  6  1  4  1  7  1  3  1  2  1  3  1  5  1
  1  2  1 15  1  5  1 11  1  2  1  3  1  2  1 23  1  5  1  4  1  2  1  3
  6  1  3  1  2  1  3  1  4  1 14  1  5  1  3  1  2  1  3  1  7  1  8  1
  1  5  1  4  1 22  1  2  1  3  1  2  1 10  1  4  1 13  1  2  1  3  1  2
  3  1  2  1  3  1  6  1  5  1  7  1  3  1  2  1  3  1  9  1  5  1  4  1
  1  7  1 10  1  2  1  3  1  2  1  4  1  6  1  5  1  2  1  3  1  2  1 11
  2  1  3  1  5  1  4  1  8  1  3  1  2  1  3  1  7  1  4  1  6  1  3  1
  1  4  1  2  1  3  1  2  1  9  1  5  1 11  1  2  1  3  1  2  1 12  1  5
  3  1  6  1 13  1  7  1  3  1  2  1  3  1  4  1  8  1  5  1  3  1  2  1
  1  2  1  3  1  2  1  5  1  4  1 15  1  2  1  3  1  2  1 10  1  4  1  9
  8  1  5  1  4  1  3  1  2  1  3  1  7  1  5  1  6  1  3  1  2  1  3  1
  1  3  1  2  1 11  1  6  1 10  1  2  1  3  1  2  1  4  1  7  1  5  1  2
  4  1  7  1  3  1  2  1  3  1  5  1  4  1  9  1  3  1  2  1  3  1  6  1
  1  2  1 23  1  5  1  4  1  2  1  3  1  2  1 14  1  5  1 11  1  2  1  3
  5  1  3  1  2  1  3  1  7  1  8  1  6  1  3  1  2  1  3  1  4  1 15  1
  1 10  1  4  1  9  1  2  1  3  1  2  1  5  1  4  1 22  1  2  1  3  1  2
  3  1  2  1  3  1 12  1  5  1  4  1  3  1  2  1  3  1  6  1  5  1  7  1
  1  6  1  5  1  2  1  3  1  2  1 11  1  7  1 10  1  2  1  3  1  2  1  4
  2  1  3  1  7  1  4  1  6  1  3  1  2  1  3  1  5  1  4  1  8  1  3  1
  1 11  1  2  1  3  1  2  1 13  1  5  1  4  1  2  1  3  1  2  1  9  1  5
  3  1  4  1  8  1  5  1  3  1  2  1  3  1  6  1 12  1  7  1  3  1  2  1
 \end{verbatim} 
 \normalsize
 \caption{A modified pattern on $24\times 24$ vertices.}
 \label{Pattern 24x24}
 \end{figure} 
 
 It is obvious that an $i$-strip in $D(k, t)$ is isomorphic to an $i$-strip in $D(1, t)$. Hence we apply results in \cite{Ekstein} for
 a coloring of strips in $D(k, t)$. Moreover we color vertices of all strips using the pattern on 24$\times$24 vertices made by Holub and
 Soukal in \cite{HolSo}, in which it is possible to replace color 16 (17) by 22 (23), because two vertices colored with 16 (17) are 
 in a whole strip at distance at least 24, respectively (see Fig. \ref{Pattern 24x24}).
 
 \begin{lemma}
  \label{Strip}
   Let $D(k, t)$ be a connected distance graph, $t > 24$, and $S_{i}$ its $i$-strip. Then it is possible to color $S_{i}$ using colors 
   $C = \{1, 2, ..., 14, 15, 22, 23\}$.
 \end{lemma}

 Then we can use colors 16 and 17 for coloring of bands as it is explained in the following statement.
 
 \begin{lemma}
  \label{lem-2bands}
  Let $D(k, t)$ be a connected distance graph, $t \geq 56$, and $B_{i}$, $B_{i + 25}$ its bands. Then it is possible to color $B_{i}$ 
  and $B_{i + 25}$ using colors $C = \{1, 16, 17, ..., 21, 24, 25, ..., 30\}$.
 \end{lemma}
 
 \begin{proof}
  We prove this lemma by exhibiting a repeating pattern using colors 1, 16, 17, ..., 21, 24, 25, ..., 30. The pattern was found with  
  help of a computer, it has period 144 and is given here: 
  
  \begin{small}
  \begin{verbatim}
   1,16,1,19,1,24,1,17,1,26,1,25,1,18,1,20,1,21,1,16,1,27,1,19,   
   1,17,1,28,1,29,1,24,1,18,1,30,1,16,1,20,1,21,1,17,1,19,1,25,  
   1,26,1,27,1,18,1,16,1,24,1,28,1,17,1,20,1,19,1,21,1,29,1,30,      
   1,16,1,18,1,25,1,17,1,26,1,24,1,19,1,20,1,21,1,16,1,27,1,18, (*)
   1,17,1,28,1,29,1,25,1,19,1,30,1,16,1,20,1,21,1,17,1,18,1,24,  
   1,26,1,27,1,19,1,16,1,25,1,28,1,17,1,20,1,18,1,21,1,29,1,30.      
  \end{verbatim}
  \end{small}
  
  \bigskip
  
  We color $B_{i}$ cyclically with pattern $(\ast)$ starting at the vertex $i$ and $B_{i + 25}$ cyclically with pattern $(\ast)$
  starting at the vertex $(i + 25)k + jt$ for any $j \in \{ -25, -24, ..., -6, 6, 7, ..., 25\}$. Let $D_{c}$ be a minimal distance 
  between two vertices colored with the same color $c$ in the same band. Then from the coloring of a whole band with pattern $(\ast)$ 
  we have $D_{25} = 26$, $D_{26} = 32$, $D_{27} = 30$, $D_{28} = 32$, $D_{29} = 32$ and $D_{30} = 36$. Let $u \in V(B_{i})$ and 
  $v \in V(B_{i + 25})$ be colored with the same color $c$. For $c \leq 24$, the distance between $u$ and $v$ is greater than 24. For 
  $c \geq 25$, the distance between $u$ and $v$ is at least $\mbox{min}\{|j|, D_{c} - |j|\}$ + 25 which is greater than $c$. 
  Hence we have a packing coloring of $B_{i}$ and $B_{i + 25}$. 
 \end{proof}
   
 In proofs of Theorems \ref{k odd, t even} and \ref{k even, t odd} we use the following statement proved by Goddard et al. in \cite{God}. 
 
 \begin{proposition}\emph{\textbf{\cite{God}}}
 \label{Goddard}
   For every $l \in \N$, the infinite path can be colored with colors $l, l + 1, ..., 3l + 2$.
 \end{proposition}

  Now we are ready to prove Theorems~\ref{k, t odd}, \ref{k odd, t even} and \ref{k even, t odd}.
 
 \bigskip
 
 \subsection{Proof of Theorem~\ref{k, t odd}} 
 \begin{proof} 
  Let $D_{p} = \{ -25, -24, ..., -6, 6, 7, ..., 25\}$. Let $k_{1} =  \mbox{min}\{k\mbox{ (mod }24), \zlom 24 - k \mbox{ (mod }24)\}$.
   
  Let $r$, $s$ be positive integers such that $t = 24s + r$, where $r$ is odd (since $t$ is odd) and minimal such that 
  $k_{1} \leq r \leq 33$. We prove Theorem \ref{k, t odd} even for $t \geq 24r + r$, which is in the worst case (for $r = 33$) 
  the general bound $t \geq 825$. Hence $s \geq r$ and we have $s$ disjoint strips and $r$ disjoint bands such that 
  $D(k, t) = S_{0}B_{24}S_{25}B_{49}... S_{24(r - 1) + r - 1}B_{24r + r - 1}S_{24r + r}... S_{24(s - 1) + r}$. 
      
  For odd $i = 1, 3, ..., r - k_{1} + 1$, we color the strips $S_{24(i - 1) + i - 1}$ cyclically with the pattern 
  from Fig. \ref{Pattern 24x24} starting at the vertex $24(i - 1)k + (i - 1)k$. If $r > k_{1}$, then, for even 
  $i = 2, 4, ..., r - k_{1}$, we color $S_{24(i - 1) + i - 1}$ cyclically with the pattern from Fig. \ref{Pattern 24x24} starting 
  at the vertex $24(i - 1)k + (i - 1)k - t$. 
   
  Let $k_{1}  \equiv k\mbox{ (mod }24)$. Then, for $i = r + 1, r + 2, ..., s$, we color $S_{24(i - 1) + r}$ cyclically with the pattern 
  from Fig. \ref{Pattern 24x24} starting at the vertex $24(i - 1)k + rk - k_{1}t$. If $k_{1} \neq 1$, then, for 
  $i = r - k_{1} + 2, r - k_{1} + 3, ..., r$, we color the strips $S_{24(i - 1) + i - 1}$ cyclically with the pattern 
  from Fig. \ref{Pattern 24x24} starting at the vertex $24(i - 1)k + (i - 1)k - (i - r + k_{1} - 1)t$.
    
  Let $k_{1} \equiv 24 - k\mbox{ (mod }24)$. Then, for $i = r + 1, r + 2, ..., s$, we color $S_{24(i - 1) + r}$ cyclically with the 
  pattern from Fig. \ref{Pattern 24x24} starting at the vertex $24(i - 1)k + rk + k_{1}t$. If $k_{1} \neq 1$, then, for 
  $i = r - k_{1} + 2, r - k_{1} + 3, ..., r$, we color the strips $S_{24(i - 1) + i - 1}$ cyclically with the pattern 
  from Fig. \ref{Pattern 24x24} starting at the vertex $24(i - 1)k + (i - 1)k + (i - r + k_{1} - 1)t$. Hence we have a packing 
  coloring of all $s$ disjoint strips of $D(k, t)$ using Lemma~\ref{Strip}.
   
  For $i = 1, 2, ..., r$, we color the bands $B_{24i + i - 1}$ cyclically with pattern~$(\ast)$ starting at the vertex 
  $24ik + (i - 1)k + j_{i}t$ such that $j_{i}$ is even (odd) for odd (even) $i$, respectively, for $i > 1$, 
  $j_{i} - j_{i - 1} \in D_{p}$, and for $s = r$, $|j_{r} - j_{1} + k| \mbox{ (mod }144) \in D_{p}$. 
  
  If $s > r$, then $j_{i}$ exist. Now assume that $r = s$. 
   
  If $r = s = 1$, then we set $j_{1} = 0$. Note that from $k_{1} \leq r$ we have only distance graphs $D(1, 25)$ and $D(23, 25)$.     
   
  If $ r = s = 5$, then we have distance graphs $D(k, 125)$ with $k \leq 123$ such that $k_{1} \leq r = 3$. 
  For $k = 1, 3, 5, 19, 43, 45, 47, 49, 51$, we set $j_{1} = 0$, $j_{2} = -13$, $j_{3} = -26$ and $j_{4}, j_{5}$ such that 
  $j_{4} - j_{3} \in D_{p}$ and $j_{5} = j_{3}$. \zlom For $k = 21, 23, 25, 27, 29, 53, 67, 69, 71$, we set $j_{1} = 0$, $j_{2} = -23$, 
  $j_{3} = -46$ and  $j_{4}, j_{5}$ such that $j_{4} - j_{3} \in D_{p}$ and $j_{5} = j_{3}$. For 
  $k = 73, 75, 77, 91, 115, 117, 119, 121\carka 123$ we set $j_{1} = 0$, $j_{2} = -25$, $j_{3} = -50$, $j_{4} = -75$ and $j_{5} = -98$. 
  For $k = 93, 95, 97, 99, 101$ we set $j_{1} = 0$, $j_{2} = -21$, $j_{3} = -42$, $j_{4} = -63$ and $j_{5} = -84$. 

  If $r = s = 3$, then we have distance graphs $D(k, 75)$ with $k \leq 73$ such that $k_{1} \leq r = 3$ and we proceed 
  for feasible $k$ in a similar way as in the case $r = s = 5$.  
   
  If $r = s \geq 7$, then we proceed analogously to previous cases (a combination of $r$ numbers from $D_{p}$ could be from 0 to 144,
  which is the length of the pattern~$(\ast)$).
   
  Hence we have a packing coloring of all $r$ disjoint bands of $D(k, t)$ using the same principle as in the proof of 
  Lemma~\ref{lem-2bands}. 
   
  Note that the bands are colored with colors $1, 16, 17, ..., 21, 24, 25, ..., 30$ and the strips are colored with colors 
  $1, 2, ..., 15, 22, 23$ in such a way that no pair of adjacent vertices is colored with color 1. Then we conclude that we have 
  a packing coloring of $D(k, t)$, hence $\chi_{\rho}(D(k, t)) \leq 30$. 
 \end{proof}
 
 \bigskip 
 
 \bigskip
    
 Some cases, in which we can decrease $t$ for which Theorem \ref{k, t odd} is true, are given in Table~\ref{tb4}.
   
 \begin{table}[ht]
  \centering
   $$ \begin{array}{|c|c|c|c|c|c|c|c|c|c|}
  \hline
   r & 1 & 3 & 5 & 7 & 9 & 11 & 13 & 15 & 17 \\
  \hline
   k_{1} & 1 & 1, 3 & 1, 3, 5 & 1, 3 , 5, 7 & 1, 3, 5, 7, 9 & \forall & \forall & \forall & \forall \\
  \hline
   t \geq & 25 & 75 & 125 & 175 & 225 & 275 & 325 & 375 & 425 \\ \hline
  \end{array}$$
  $$ \begin{array}{|c|c|c|c|c|c|c|c|c|}
  \hline
   r & 19 & 21 & 23 & 25 & 27 & 29 & 31 & 33 \\
  \hline
   k_{1} & \forall & \forall & \forall & 3, 5, 7, 9, 11 & 5, 7, 9, 11 & 7, 9, 11 & 9, 11 & 11 \\
  \hline
   t \geq & 475 & 525 & 575 & 625 & 675 & 725 & 775 & 825  \\ \hline
  \end{array}$$
  \caption{\label{tb4} Table for $t$ depending on odd $k_{1}, r$ with $r \geq k_{1}$.}
  \normalsize
 \end{table}
 
 \bigskip
   
 \subsection{Proof of Theorem~\ref{k odd, t even}} 
 \begin{proof} 
  Let $D_{p} = \{ -25, -24, ..., -6, 6, 7, ..., 25\}$. Let $k_{1} =  \mbox{min}\{k\mbox{ (mod }24), \zlom 24 - k \mbox{ (mod }24)\}$.
   
  Let $r$, $s$ be positive integers such that $t = 24(s + 2) + r$, where $r$ is even (since $t$ is even) and minimal such that 
  $k_{1} < r \leq 34$. We prove Theorem \ref{k odd, t even} even for $t \geq 24(r + 2) + r$, which is in the worst case (for $r = 34$) 
  the general bound $t \geq 898$. Hence $s \geq r$ and we have $s + 2$ disjoint strips and $r$ disjoint bands such that 
  $D(k, t) = S_{0}S_{24}B_{48}S_{49}...S_{24(r - 1) + r - 2}B_{24r + r - 2}S_{24r + r - 1}$
  $S_{24(r + 1) + r - 1}... S_{24(s + 1) + r - 1} B_{24(s + 2) + r - 1}$.
  
  We color the strip $S_{0}$ cyclically with the pattern from Fig. \ref{Pattern 24x24} starting at the vertex 0. 
  For odd $i = 1, 3, ..., r - k_{1}$, we color the strips $S_{24i + i - 1}$ cyclically with the pattern from Fig. \ref{Pattern 24x24}
  starting at the vertex $24ik + (i - 1)k$. If $r > k_{1} + 1$, then, for even $i = 2, 4, ..., r - k_{1} - 1$, we color $S_{24i + i - 1}$
  cyclically with the pattern from Fig. \ref{Pattern 24x24} starting at the vertex $24ik + (i - 1)k - t$. 
   
  Let $k_{1}  \equiv k\mbox{ (mod }24)$. Then, for $i = r + 1, ..., s + 1$, we color $S_{24i + r - 1}$ cyclically with the 
  pattern from Fig. \ref{Pattern 24x24} starting at the vertex $24ik + (r - 1)k - k_{1}t$. If $k_{1} \neq 1$, then, for 
  $i = r - k_{1} + 1, r - k_{1} + 2, ..., r$, we color the strips $S_{24i + i - 1}$ cyclically with the pattern from 
  Fig. \ref{Pattern 24x24} starting at the vertex $24ik + (i - 1)k - (i - r + k_{1})t$.
   
  Let $k_{1} \equiv 24 - k\mbox{ (mod }24)$. Then, for $i = r + 1, r + 2, ..., s + 1$, we color $S_{24i + r - 1}$ cyclically with the 
  pattern from Fig. \ref{Pattern 24x24} starting at the vertex $24ik + (r -1)k + k_{1}t$. If $k_{1} \neq 1$, then, for 
  $i = r - k_{1} + 1, r - k_{1} + 2, ..., r$, we color the strips $S_{24i + i - 1}$ cyclically with the pattern from 
  Fig. \ref{Pattern 24x24} starting at the vertex $24ik + (i - 1)k + (i - r + k_{1})t$. Hence we have a packing coloring of 
  all $s$ disjoint strips of $D(k, t)$ using Lemma~\ref{Strip}.
   
  For $i = 1, 2, ..., r - 1$, we color the bands $B_{24(i + 1) + i - 1}$ cyclically with pattern~$(\ast)$ starting at the vertex 
  $24(i + 1)k + (i - 1)k + j_{i}t$ such that $j_{i}$ is even (odd) for odd (even) $i$, respectively, and for $i > 1$, 
  $j_{i} - j_{i - 1} \in D_{p}$. 
    
  We color $B_{24(s + 2) + r - 1}$ with a sequence of colors 18, 19, ..., 21, 16, 17, 24, 25, ..., 56 starting at any vertex 
  of $B_{24(s + 2) + r - 1}$. By Proposition \ref{Goddard} for $l = 18$, we can color  $B_{24(s + 2) + r - 1}$ with colors 
  18, 19, ..., 56. Since we used colors 22 and 23 for a coloring of strips, we replace color 22 (23) by 16 (17), respectively,
  in the coloring of this band. Note the band $B_{24(s + 2) + r - 1}$ is the only one with colors greater than 48. 
  By Lemma \ref{lem-2bands} and the fact that a distance between any vertex of  $B_{24(s + 2) + r - 1}$ and any vertex of any 
  other band of $D(k, t)$  is at least 49 the defined coloring is a packing coloring of all $r$ disjoint bands of $D(k, t)$.    
   
  Note that the bands are colored with colors $1, 16, 17, ..., 21, 24, 25, ..., 56$ and the strips are colored with colors 
  $1, 2, ..., 15, 22, 23$ in such a way that no pair of adjacent vertices is colored with color 1. Then we conclude that we have a 
  packing coloring of $D(k, t)$, hence $\chi_{\rho}(D(k, t)) \leq 56$. 
 \end{proof}
  
  Some cases, in which we can decrease $t$ for which Theorem \ref{k odd, t even} is true, are summarized in Table~\ref{tb5}.
   
 \begin{table}[ht]
  \centering
  $$ \begin{array}{|c|c|c|c|c|c|c|c|c|c|}
  \hline
   r & 2 & 4 & 6 & 8 & 10 & 12 & 14 & 16 & 18  \\
  \hline
   k_{1} & 1 & 1, 3 & 1, 3, 5 & 1, 3 , 5, 7 & 1, 3, 5, 7, 9 & \forall & \forall & \forall & \forall \\
  \hline
   t \geq & 98 & 148 & 198 & 248 & 298 & 348 & 398 & 448 & 498 \\ \hline
  \end{array}$$
  $$ \begin{array}{|c|c|c|c|c|c|c|c|c|}
  \hline
   r & 20 & 22 & 24 & 26 & 28 & 30 & 32 & 34 \\
  \hline
   k_{1} & \forall & \forall & \forall & 3, 5, 7, 9, 11 & 5, 7, 9, 11 & 7, 9, 11 & 9, 11 & 11 \\
  \hline
   t \geq & 548 & 598 & 648 & 698 & 748 & 798 & 848 & 898 \\ \hline
  \end{array}$$
  \caption{\label{tb5} Table for $t$ depending on odd $k_{1}$ and even $r$ with $r > k_{1}$.}
  \normalsize
 \end{table}
  
  \subsection{Proof of Theorem~\ref{k even, t odd}} 
 \begin{proof} 
  Let $D_{p} = \{ -25, -24, ..., -6, 6, 7, ..., 25\}$. Let $k_{1} =  \mbox{min}\{k\mbox{ (mod }24), \zlom 24 - k \mbox{ (mod }24)\}$. 

  Let $s$ be a positive integer such that $t = 24s + 1$ and $k_{1} = 0$. Hence we have $s$ disjoint strips and 1 band such that 
  $D(k, t) = S_{0}S_{24}... S_{24(s - 1)}B_{24s}$. For $i = 0, ... s - 1$ we color the strips $S_{24i}$ cyclically with the pattern from 
  Fig. \ref{Pattern 24x24} starting at the vertex $24ik$. We color $B_{24s}$ with a sequence of colors 18, 19, ..., 21, 16, 17, 24, 
  25, ..., 56 starting at any vertex of $B_{24s}$. Hence the band $B_{24s}$ is colored with colors $1, 16, 17, ..., 21, 24, 25, ..., 56$
  and the strips are colored with colors $1, 2, ..., 15, 22, 23$ in such a way that no pair of adjacent vertices is colored with color 1.
  Then we conclude that we have a packing coloring of $D(k, t)$, hence $\chi_{\rho}(D(k, t)) \leq 56$. 
    
  Let $r$, $s$ be positive integers such that $t = 24(s + 2) + r$, where $r$ is odd (since $t$ is odd) and minimal such that 
  $k_{1} < r \leq 35$. We exclude the previous case $k_{1} = 0$, $r = 1$ and we prove Theorem \ref{k even, t odd} even for 
  $t \geq 24(r + 2) + r$, which is in the worst case (for $r = 35$) the general bound $t \geq 923$. Hence $s \geq r$ and we have 
  $s + 2$ disjoint strips and $r$ disjoint bands such that $D(k, t) = S_{0}S_{24}B_{48}S_{49}...S_{24(r - 1) + r - 2}
  B_{24r + r - 2}S_{24r + r - 1}S_{24(r + 1) + r - 1}...$ $S_{24(s + 1) + r - 1}B_{24(s + 2) + r - 1}$.  
    
  The rest of the proof of Theorem \ref{k even, t odd} is exactly same as of the proof of Theorem~\ref{k odd, t even}.
  \end{proof}
  
  Some cases, in which we can decrease $t$ for which Theorem \ref{k even, t odd} is true, are given in Table~\ref{tb6}.
   
 \begin{table}[ht]
  \centering
  \footnotesize
  $$ \begin{array}{|c|c|c|c|c|c|c|c|c|c|c|}
  \hline
   r & 1 & 3 & 5 & 7 & 9 & 11 & 13 & 15 & 17 & 19 \\
  \hline
   k_{1} & 0 & 0, 2 & 0, 2, 4 & 0, 2, 4, 6 & 0, 2, 4, 6, 8 & 0, 2, 4, 6, 8, 10 & \forall & \forall & \forall & \forall \\
  \hline
   t \geq & 25 & 123 & 173 & 223 & 273 & 323 & 373 & 423 & 473 & 523 \\ \hline
  \end{array}$$
  $$ \begin{array}{|c|c|c|c|c|c|c|c|c|}
  \hline
   r  & 21 & 23 & 25 & 27 & 29 & 31 & 33 & 35 \\
  \hline
   k_{1} & \forall & \forall & 2, 4, 6, 8, 10, 12 & 4, 6, 8, 10, 12 & 6, 8, 10, 12 & 8, 10, 12 & 10, 12 & 12 \\
  \hline
   t \geq & 573 & 623 & 673 & 723 & 773 & 823 & 873 & 923 \\ \hline
  \end{array}$$
  \caption{\label{tb6} Table for $t$ depending on even $k_{1}$ and odd $r$ with $r > k_{1}$.}
  \normalsize
 \end{table}
   
  \subsection{Proof of Theorem~\ref{general lower bound}}   
  
  \begin{proof}
   It is shown in \cite{EksteinFiala} that a finite square lattice $15 \times 9$ cannot be colored using $11$
   colors. Clearly $D(k, t)$ contains a finite square grid as a subgraph and $t \geq 9$ assures existence of the square
   lattice $15 \times 9$ in a connected $D(k, t)$. Therefore,  $\chi_{\rho}(D(k, t)) \geq 12$ for every $t \geq 9$.
  \end{proof}  

 \section{Remarks and acknowledgemets}

  The access to the METACentrum computing facilities, provided under the programme "Projects of Large Infrastructure for Research, 
  Development and Innovations" LM2010005 funded by the Ministry of Education, Youth and Sports of the Czech Republic, is highly 
  appreciated.
  
  This work was supported by the European Regional Development Fund (ERDF), project “NTIS - New Technologies for Information Society”, 
  European Centre of Excellence, CZ.1.05/1.1.00/02.0090.
  
  First two authors were supported by the Center of Excellence -- Inst. for Theor. Comp. Sci., Prague (project P202/12/G061 of 
  GA~\v{C}R).
  
  The third author was supported by the University of Burgundy (project BQR 036, 2011).


\begin{thebibliography}{10}
  \bibitem{Bon} J.~A.~Bondy, U.~S.~R.~Murty,
   \emph{Graph Theory with Applications}, Macmillan, London and Elsevier (1976).
   
  \bibitem{Bresar} B.~Bre\v{s}ar, S.~Klav\v{z}ar, D.~F.~Rall,
   \emph{On the packing chromatic number of Cartesian products, hexagonal lattice and trees},
    Discrete Appl. Math. {\bf 155} (2007) 2303-2311.   
   
  \bibitem{ChaHuZhu} G.~J.~Chang, L.~Huang, X.~Zhu, 
    \emph{The circular chromatic numbers and the fractional chromatic numbers of distance graphs}, 
    European Journal Combin. {\bf 19} (1998) 423-431.

  \bibitem{EksteinFiala} J.~Ekstein, J.~Fiala, P.~Holub, B.~Lidick\'{y},
    \emph{The packing chromatic number of the square lattice is at least 12}, arXiv: 1003.2291v1, preprint (2010).
  
  \bibitem{Ekstein} J.~Ekstein, P.~Holub, B.~Lidick\'{y},
    \emph{Packing chromatic number of distance graphs}, Discrete Appl. Math. {\bf 160} (2012) 518-524.
    
  \bibitem{Eggl} R.~B.~Eggleton, P.~Erd\"{o}s, D.~K.~Skilton,
     \emph{Colouring the real line}, J.~Combin. Theory Ser. B {\bf 39 (1)} (1985) 86-100.

  \bibitem{FiaGol} J.~Fiala, P.~A.~Golovach,
     \emph{Complexity of the packing chromatic problem for trees}, Discrete Appl. Math. {\bf 158} (2010) 771-778.

  \bibitem{bib-fiala09+} J.~Fiala, S.~Klav\v{z}ar, B.~Lidick\'{y},
     \emph{The packing chromatic number of infinite product graphs}, European J. Combin. {\bf 30 (5)} (2009) 1101-1113.

  \bibitem{God} W.~Goddard, S.~M.~Hedetniemi, S.~T.~Hedetniemi, J.~M.~Harris, D.~F.~Rall,
     \emph{Broadcast chromatic numbers of graphs}, Ars Combin. {\bf 86} (2008) 33-49.

  \bibitem{HolSo} P.~Holub, R.~Soukal, 
     \emph{A note on packing chromatic number of the square lattice}, Electronic Journal of Combinatorics {\bf 17} (2010) Note 17.
  
  \bibitem{KemKo} A. Kemnitz, H. Kolberg, 
      \emph{Coloring of integer distance graphs}, Discrete Math. {\bf 191} (1998) 113-123.  
  
  \bibitem{LiLaSo} W. Lin, P. Lam, Z. Song, 
     \emph{Circular chromatic numbers of some distance graphs}, Discrete Math. {\bf 292} (2005) 119-130.  
  
  \bibitem{RuzTuVoi} I.~Z.~Ruzsa, Z.~Tuza, M.~Voigt,
      \emph{Distance graphs with finite chromatic number}, J. Combin. Theory Ser. B {\bf 85 (1)} (2002) 181-187.

  \bibitem{Togni} O.~Togni,
      \emph{On Packing Colorings of Distance Graphs}, preprint (2010).
      
  \bibitem{Zhu}  X. Zhu, 
      \emph{The circular chromatic number of a class of distance graphs}, Discrete Math. {\bf 265 (1-3)} (2003) 337-350.    

\end{thebibliography}
\end{document}